\documentclass[11pt]{amsart}

\usepackage{amssymb}
\usepackage{mathtools}
\usepackage[numbers,sort&compress]{natbib}
\usepackage{xcolor}
\usepackage[margin=1.4in]{geometry}
\usepackage{hyperref}

\hypersetup{
  pdfauthor={Bonan Chen},
  pdftitle={Endpoint and Vanishing-Density Asymptotics for Hardy--Szego Zero Counts},
  pdfsubject={Endpoint and vanishing-density asymptotics for Hardy--Szego zero counts},
  pdfkeywords={Hardy--Szego zeros, determinantal point processes, fixed-count asymptotics, hole probabilities, vanishing-density asymptotics}
}

\title[Endpoint and Vanishing-Density Asymptotics]
{Endpoint and Vanishing-Density Asymptotics for Hardy--Szeg\H{o} Zero Counts}

\author{Bonan Chen}
\address{School of Mathematical Sciences, Soochow University,
Suzhou 215006, P. R. China}
\email{bnchen@suda.edu.cn}

\subjclass[2020]{60G55, 60F10, 47B35}
\keywords{Hardy--Szeg\H{o} zeros; determinantal point processes;
fixed-count asymptotics; hole probabilities; vanishing-density asymptotics}

\numberwithin{equation}{section}

\theoremstyle{plain}
\newtheorem{theorem}{Theorem}[section]
\newtheorem{lemma}[theorem]{Lemma}
\newtheorem{proposition}[theorem]{Proposition}

\theoremstyle{definition}

\theoremstyle{remark}

\begin{document}

\begin{abstract}
We study the number \(N_I(L)\) of zeros of the Hardy--Szeg\H{o} zero process in the horizontal window \([0,L]\times I\) as \(L\to\infty\), where \(I=[\alpha,\beta]\Subset(0,\infty)\) is fixed.  
We obtain sharp point-probability asymptotics at the lower endpoint of the density scale.  
In the fixed-count regime, for every fixed integer \(k\geq0\), we determine a full asymptotic formula for \(\mathbb P\{N_I(L)=k\}\), identifying its exponential rate, order-one correction, and \(k\)-dependent polynomial prefactor; the case \(k=0\) gives the hole probability.
In the vanishing-density regime, we prove a uniform local asymptotic formula for \(b_L\leq n\leq\varepsilon_L L\), where \(b_L\to\infty\), \(\varepsilon_L\downarrow0\), and \(b_L\leq\varepsilon_L L\), identifying the large-deviation exponent, endpoint correction, and Gaussian prefactor.
\end{abstract}

\maketitle

\section{Introduction and main results}\label{sec:introduction}

For point processes in growing observation windows, a natural objective is to understand the distribution of the number of points across its different asymptotic scales.  
In the determinantal and Gaussian-analytic settings most relevant here, different aspects of this problem have been studied in several regimes.  
At the central scale, Gaussian fluctuations of determinantal counting statistics are classical; see, for example, \cite{CostinLebowitz1995,Soshnikov2002}, while local Gaussian asymptotics describe the corresponding individual lattice probabilities near the typical count \cite{ForresterLebowitz2014}.  
Beyond the central scale, precise deviations and extreme low-count events, including hole probabilities, have also been studied in Gaussian-analytic and determinantal models; see, for example, \cite{BuckleyNishryPeledSodin2018,FenzlLambert2022,Nishry2010,SodinTsirelson2005}.

\paragraph*{The Hardy--Szeg\H{o} window count.} We study this multi-scale counting problem for the Hardy--Szeg\H{o} zero process in long bounded-height windows of the upper half-plane.  
The process is the upper-half-plane realization, under a Cayley transformation, of the zero set of the parameter-one hyperbolic Gaussian analytic function.  
In the unit disk, Peres and Vir\'ag proved that this zero set is a conformally invariant determinantal point process with the Bergman correlation kernel; see \cite[Theorem~1 and Proposition~9]{PeresVirag2005}, and see \cite[Chapters~2--5]{HoughKrishnapurPeresVirag2009} for general background.
Thus the model combines the geometric and conformal structure of Gaussian-analytic zeros with the repulsive and operator-theoretic structure of a determinantal point process.
In particular, the determinantal representation allows counting generating functions to be expressed as Fredholm determinants, making the process especially well suited to precise counting questions.

In this upper-half-plane realization, the Hardy--Szeg\H{o} zero process is a determinantal point process \(Z_{\mathbb H}\) on \( \mathbb H=\{z\in\mathbb C:\operatorname{Im}z>0\} \) with correlation kernel
\[
K_{\mathbb H}(z,w)
=
-\frac{1}{\pi(z-\overline w)^2}
\]
with respect to planar Lebesgue measure.
It is stationary under horizontal translations.
For an interval \(I=[\alpha,\beta]\) with \(0<\alpha<\beta<\infty\), we consider the count
\[
N_I(L)
=
\#\bigl(Z_{\mathbb H}\cap([0,L]\times I)\bigr).
\]
Its one-point intensity gives
\[
\mathbb E N_I(L)
=
\frac{L}{4\pi}
\left (\frac{1}{\alpha}-\frac{1}{\beta}\right ),
\]
so \(N_I(L)/L\) is the natural density variable for the multi-scale counting problem posed above.

\paragraph*{The zero-density endpoint.}
The work \cite{AiFangHou2026} establishes the large-deviation principle for \(N_I(L)/L\) and the strong Szeg\H{o} asymptotics that provide the analytic framework used here.
The present paper extends this analysis to the zero-density endpoint and derives precise point-probability asymptotics in two regimes.
The first consists of fixed counts,
\[
\mathbb P\{N_I(L)=k\},
\qquad k\geq0\ \text{fixed},
\]
including the hole event \(k=0\).
The second consists of integer sequences \(n=n_L\) satisfying
\[
n_L\longrightarrow\infty,
\qquad
\frac{n_L}{L}\longrightarrow0.
\]
The resulting asymptotics describe both the exact counting probabilities at the endpoint and their uniform behavior along diverging counts whose density tends to zero.

\paragraph*{The endpoint mechanism.}
The probability generating function \(\mathbb E z^{N_I(L)}\) admits a one-dimensional Fredholm-determinant representation whose large-\(L\) analysis is based on the Wiener--Hopf comparison developed in \cite{AiFangHou2026}.
Fixed counts are obtained from coefficients at \(z=0\), while the vanishing-density saddle approaches the same point, so both regimes require precise control near \(z=0\), equivalently near \(t=-1\) with \(t=z-1\).

\subsection{Main results}\label{subsec:main-results}

We first introduce the notation for the endpoint expansion.
For the interval \(I=[\alpha,\beta]\) fixed above, set
\[
q_I(\xi)
=
e^{-2\alpha\xi}-e^{-2\beta\xi},
\qquad \xi>0,
\]
and define
\[
\rho_I
:=
\sup_{\xi>0}q_I(\xi),
\qquad
\Omega_I
:=
\mathbb C\setminus(-\infty,-\rho_I^{-1}].
\]
Since \(0<\rho_I<1\), the domain \(\Omega_I\) contains \([-1,\infty)\).
With the logarithm normalized at \(t=0\), let
\[
\mathcal L_I(t)
=
\frac{1}{2\pi}
\int_0^\infty
\log\bigl(1+tq_I(\xi)\bigr)\,\mathrm d\xi,
\qquad t\in\Omega_I.
\]

The following theorem extends the strong Szeg\H{o} expansion of \cite[Corollary~1.5]{AiFangHou2026} to \(\Omega_I\), which contains the endpoint \(t=-1\), and provides the common analytic input for the two probability results below.

\begin{theorem}
\label{thm:endpoint-expansion}
For each \(L>0\), the function \(t\mapsto\mathbb E(1+t)^{N_I(L)}\) is zero-free on \(\Omega_I\).
There exists a unique holomorphic function \(D_I\) on \(\Omega_I\), with \(D_I(0)=0\), such that
\[
\log\mathbb E(1+t)^{N_I(L)}
=
L\mathcal L_I(t)+D_I(t)+o(1),
\qquad L\to\infty,
\]
locally uniformly on \(\Omega_I\), where the logarithm denotes the holomorphic branch vanishing at \(t=0\).
Moreover, \(D_I\) is real-valued on \((-\rho_I^{-1},\infty)\).
\end{theorem}

\paragraph*{Fixed-count asymptotics.} Define
\begin{equation}
\label{eq:endpoint-rate}
J_I(0)
:=
-\mathcal L_I(-1)
=
-\frac{1}{2\pi}
\int_0^\infty
\log\bigl(1-q_I(\xi)\bigr)\,\mathrm d\xi,
\end{equation}
and
\[
\kappa_I
:=
\frac{1}{2\pi}
\int_0^\infty
\frac{q_I(\xi)}{1-q_I(\xi)}\,\mathrm d\xi
>0.
\]

\begin{theorem}
\label{thm:fixed-counts}
For every fixed integer \(k\geq0\),
\[
\mathbb P\{N_I(L)=k\}
=
\exp\{-LJ_I(0)+D_I(-1)\}
\frac{(L\kappa_I)^k}{k!}
(1+o(1)),
\qquad L\to\infty.
\]
\end{theorem}

The case \(k=0\) gives the precise hole-probability asymptotic.
All fixed counts share the exponential rate \(J_I(0)\) and the correction \(D_I(-1)\), with their dependence on \(k\) captured by the prefactor \((L\kappa_I)^k/k!\).

\paragraph*{Vanishing-density asymptotics.} To describe diverging counts of vanishing density, write \(s=e^\theta=1+t\), so that the endpoint \(t=-1\) corresponds to \(s=0\), and define
\begin{equation}\label{eq:endpoint-A}
A_I(s)
:=
\mathcal L_I(s-1)-\mathcal L_I(-1)
=
\frac{1}{2\pi}
\int_0^\infty
\log\left(
1+s\frac{q_I(\xi)}{1-q_I(\xi)}
\right)\,\mathrm d\xi,
\qquad s\geq0.
\end{equation}
Then
\[
A_I(0)=0,
\qquad
A_I'(0)=\kappa_I.
\]
By \cite[Theorem~1.2 and Corollary~1.5]{AiFangHou2026}, the good large-deviation rate function for \(N_I(L)/L\) can equivalently be written as
\begin{equation}
\label{eq:rate-variational}
J_I(x)
=
J_I(0)
+
\sup_{s>0}
\bigl\{x\log s-A_I(s)\bigr\},
\qquad x\geq0.
\end{equation}
For every \(x>0\), the supremum in \eqref{eq:rate-variational} is attained at a unique point \(s>0\), characterized by
\[
sA_I'(s)=x;
\]
see \cite[Corollary~4.3]{AiFangHou2026}.
For the target density \(x=n/L\), denote this saddle point by \(s_{L,n}\).
Thus
\begin{equation}
\label{eq:limiting-saddle-equation}
s_{L,n}A_I'(s_{L,n})
=
\frac{n}{L}.
\end{equation}

\begin{theorem}
\label{thm:vanishing-density}
Let \(b_L,\varepsilon_L>0\) satisfy \(b_L\to\infty\), \(\varepsilon_L\to0\), and \(b_L\leq\varepsilon_L L\) for all sufficiently large \(L\).
Then, as \(L\to\infty\), uniformly over integers \(b_L\leq n\leq\varepsilon_L L\),
\[
\mathbb P\{N_I(L)=n\}
=
\frac{
\exp\{-LJ_I(n/L)+D_I(s_{L,n}-1)\}
}{
\sqrt{
2\pi L\bigl(
s_{L,n}A_I'(s_{L,n})
+s_{L,n}^2A_I''(s_{L,n})
\bigr)
}
}
(1+o(1)).
\]
Moreover, uniformly over the same integers,
\[
s_{L,n}\sim\frac{n}{L\kappa_I},
\qquad
L\bigl(
s_{L,n}A_I'(s_{L,n})
+s_{L,n}^2A_I''(s_{L,n})
\bigr)
\sim n.
\]
\end{theorem} 

\subsection{Proof strategy and organization}
\label{subsec:proof-strategy}
\label{subsec:introduction-organization}

The proof of Theorem~\ref{thm:endpoint-expansion} starts from the one-dimensional Fredholm representation \eqref{eq:general-interval-fredholm}, which expresses the probability generating function in terms of the reduced positive trace-class operator \(T_{I,L}\) on \(L^2(0,\infty)\).
The strict spectral bound \eqref{eq:strict-spectral-bound} then gives \(\sigma(T_{I,L})\subset[0,\rho_I]\), uniformly in \(L\), and hence
\[
\det\bigl(\operatorname{Id}+tT_{I,L}\bigr)\neq0,
\qquad
t\in\Omega_I,
\]
with \(-1\in\Omega_I\).
The strong Szeg\H{o} expansion in \cite[Corollary~1.5]{AiFangHou2026} gives convergence on the interior domain.
By extending the Wiener--Hopf estimates and comparison argument developed there, we establish local boundedness on \(\Omega_I\) of the normalized determinants
\[
e^{-L\mathcal L_I(t)}
\det\bigl(\operatorname{Id}+tT_{I,L}\bigr).
\]
A normal-family argument, based on Vitali's theorem, extends their interior convergence to \(\Omega_I\), while Hurwitz's theorem shows that the limiting normalized determinant remains zero-free.
Its normalized holomorphic logarithm defines the endpoint correction \(D_I\).
This gives Proposition~\ref{prop:endpoint-determinant-general} and, together with \eqref{eq:general-interval-fredholm}, proves Theorem~\ref{thm:endpoint-expansion}.

Theorem~\ref{thm:fixed-counts} follows from two uses of this endpoint expansion.
Evaluation of Proposition~\ref{prop:endpoint-determinant-general} at \(t=-1\) gives the hole probability.
For \(k\geq1\), the change of variables \(t=z-1\) identifies \(t=-1\) with \(z=0\).
Proposition~\ref{prop:endpoint-neighborhood} gives a locally uniform holomorphic expansion near \(z=0\) of the normalized probability generating function
\[
\frac{\mathbb E z^{N_I(L)}}{\mathbb P\{N_I(L)=0\}}.
\]
Cauchy's coefficient formula on circles \(|z|=R/L\), with \(R>0\) fixed, then controls \([z^k]\) and yields the factorial prefactor in Theorem~\ref{thm:fixed-counts}.

For Theorem~\ref{thm:vanishing-density}, let \(\mathbb Q_{L,r}\) denote the exponential tilt of \(N_I(L)\) with parameter \(r>0\).
The tilted count again has a determinantal counting law, with correlation eigenvalues obtained by an explicit transform of the eigenvalues of \(T_{I,L}\).
For each target count \(n\), the exact finite-\(L\) saddle \(r_{L,n}>0\) is defined by
\[
\mathbb E_{\mathbb Q_{L,r_{L,n}}}N_I(L)=n.
\]
Proposition~\ref{prop:endpoint-saddle-variance} constructs this saddle and shows that the corresponding variance is asymptotic to \(n\).
The local central limit theorem \cite[Theorem~1]{ForresterLebowitz2014}, combined with the exact change-of-measure identity, then yields the finite-saddle probability asymptotic in Proposition~\ref{prop:exact-saddle-point-probability}.

Lemma~\ref{lem:saddle-point-replacement} compares the exact and limiting saddle points and shows that the replacement preserves the exponent and variance at the required precision.
Legendre duality then identifies the exponent with \(J_I(n/L)\), completing the proof of Theorem~\ref{thm:vanishing-density}.

Section~\ref{sec:endpoint-fixed-counts} proves the endpoint expansion and the fixed-count asymptotics.
Section~\ref{sec:vanishing-density} combines exponential tilting, the determinantal local central limit theorem, and finite- and limiting-saddle analysis to prove the vanishing-density asymptotics.

\section{Endpoint and fixed-count asymptotics}
\label{sec:endpoint-fixed-counts}

In this section, we prove Theorems~\ref{thm:endpoint-expansion} and~\ref{thm:fixed-counts}.
We first use a strict spectral bound and a normal-family argument to extend the strong Szeg\H{o} expansion of \cite[Corollary~1.5]{AiFangHou2026} to the zero-free domain \(\Omega_I\).
We then derive the fixed-count asymptotics by endpoint evaluation at \(t=-1\) and coefficient extraction near \(z=0\).

\subsection{Endpoint determinant expansion}

Throughout this section, we equip the first-order Sobolev space \(W^{1,2}(\mathbb R)\) with the norm
\[
\|f\|_{W^{1,2}(\mathbb R)}
:=
\left(
\|f\|_{L^2(\mathbb R)}^2
+
\|f'\|_{L^2(\mathbb R)}^2
\right)^{1/2},
\]
where \(f'\) denotes the weak derivative.
\(\mathfrak S_1\) and \(\mathfrak S_2\) denote the trace-class and Hilbert--Schmidt ideals, respectively, with norms \(\|\cdot\|_{\mathfrak S_1}\) and \(\|\cdot\|_{\mathfrak S_2}\).
 For \(T\in\mathfrak S_1\), \(\det(\operatorname{Id}+T)\) denotes the corresponding Fredholm determinant.
We refer to \cite{Simon2005} for standard background on trace ideals and Fredholm determinants.

Define \(T_{I,L}\) to be the integral operator on \(L^2(0,\infty)\) with kernel
\[
T_{I,L}(\xi,\eta)
=
\frac{\sqrt{\xi\eta}}{\pi}
\left(\int_0^L e^{-\mathrm i(\xi-\eta)x}\,\mathrm dx\right)
\left(\int_\alpha^\beta e^{-(\xi+\eta)y}\,\mathrm dy\right).
\]
By the dilation reduction and the argument of \cite[Lemmas~2.1 and~2.6]{AiFangHou2026}, \(T_{I,L}\) is positive trace class and satisfies
\begin{equation}\label{eq:general-interval-fredholm}
\mathbb E z^{N_I(L)}
=
\det\bigl(\operatorname{Id}+(z-1)T_{I,L}\bigr),
\qquad z\in\mathbb C.
\end{equation}
The cited result proves the identity for \(z\neq0\); since both sides are continuous at \(z=0\), it also holds there, where it reads
\[
\mathbb P\{N_I(L)=0\}
=
\det\bigl(\operatorname{Id}-T_{I,L}\bigr).
\]

We retain the notation \(q_I\), \(\rho_I\), \(\Omega_I\), and \(\mathcal L_I\) from Section~\ref{subsec:main-results}.
We first establish a uniform spectral bound for \(T_{I,L}\), which provides the zero-free domain needed to extend the interior strong Szeg\H{o} expansion to the endpoint \(t=-1\).

\begin{lemma}
\label{lem:strict-spectral-bound}
For every \(L>0\),
\begin{equation}\label{eq:strict-spectral-bound}
0\leq T_{I,L}\leq\rho_I\operatorname{Id}.
\end{equation}
Consequently, \(\det(\operatorname{Id}+tT_{I,L})\neq0\) for every \(t\in\Omega_I\).
\end{lemma}

\begin{proof}
For \(\xi>0\) and \(z\in\mathbb H\), set
\[
u_\xi(z)=\sqrt{{\xi}/{\pi}}\,e^{\mathrm i\xi z},
\qquad
(\mathsf A_{I,L}f)(\xi)
=
\int_{[0,L]\times I}\overline{u_\xi(z)}f(z)\,\mathrm dA(z).
\]
The operator \(\mathsf A_{I,L}:L^2([0,L]\times I)\to L^2(0,\infty)\) is Hilbert--Schmidt, since
\[
\|\mathsf A_{I,L}\|_{\mathfrak S_2}^2
=
\int_{[0,L]\times I}\int_0^\infty |u_\xi(z)|^2
\,\mathrm d\xi\,\mathrm dA(z)
=
\frac{L}{4\pi}\left(\frac1\alpha-\frac1\beta\right).
\]
Moreover,
\[
\int_{[0,L]\times I}\overline{u_\xi(z)}u_\eta(z)\,\mathrm dA(z)
=
\frac{\sqrt{\xi\eta}}{\pi}
\left(\int_0^L e^{-\mathrm i(\xi-\eta)x}\,\mathrm dx\right)
\left(\int_\alpha^\beta e^{-(\xi+\eta)y}\,\mathrm dy\right),
\]
and hence \( T_{I,L}=\mathsf A_{I,L}\mathsf A_{I,L}^*. \)

To estimate \(\mathsf A_{I,L}\), extend \(f\in L^2([0,L]\times I)\) by zero to \(\mathbb R\times I\), and set
\[
\widehat f(\omega,y)
=
\frac{1}{\sqrt{2\pi}}
\int_{\mathbb R}e^{-\mathrm i\omega x}f(x,y)\,\mathrm dx.
\]
For \(\omega>0\),
\[
\begin{aligned}
(\mathsf A_{I,L}f)(\omega)
&=
\sqrt{\frac{\omega}{\pi}}
\int_\alpha^\beta e^{-\omega y}
\left(\int_0^L e^{-\mathrm i\omega x}f(x,y)\,\mathrm dx\right)
\mathrm dy \\
&=
\sqrt{2\omega}
\int_\alpha^\beta e^{-\omega y}\widehat f(\omega,y)\,\mathrm dy.
\end{aligned}
\]
Therefore, by the Cauchy--Schwarz inequality,
\[
|(\mathsf A_{I,L}f)(\omega)|^2
\leq
\left(2\omega\int_\alpha^\beta e^{-2\omega y}\,\mathrm dy\right)
\|\widehat f(\omega,\cdot)\|_{L^2(I)}^2
\leq
\rho_I\|\widehat f(\omega,\cdot)\|_{L^2(I)}^2.
\]
Integrating and applying Plancherel's theorem gives
\[
\|\mathsf A_{I,L}f\|_{L^2(0,\infty)}^2
\leq
\rho_I\int_{\mathbb R}
\|\widehat f(\omega,\cdot)\|_{L^2(I)}^2\,\mathrm d\omega
=
\rho_I\|f\|_{L^2([0,L]\times I)}^2.
\]
Thus \(\|\mathsf A_{I,L}\|^2\leq\rho_I\), and consequently
\[
0\leq T_{I,L}
=
\mathsf A_{I,L}\mathsf A_{I,L}^*
\leq
\|\mathsf A_{I,L}\|^2\operatorname{Id}
\leq
\rho_I\operatorname{Id}.
\]

Finally, if \(\det(\operatorname{Id}+tT_{I,L})=0\) for some \(t\in\Omega_I\), then \(1+t\lambda=0\) for some nonzero eigenvalue \(\lambda\) of \(T_{I,L}\).
Since \(0<\lambda\leq\rho_I\), this would give \(t=-\lambda^{-1}\in(-\infty,-\rho_I^{-1}]\), contradicting \(t\in\Omega_I\).
\end{proof}

By the general-interval strong Szeg\H{o} expansion \cite[Corollary~1.5]{AiFangHou2026}, there exists a holomorphic function \(\mathcal C_I\) on \(\{\theta\in\mathbb C:|\operatorname{Im}\theta|<\pi\}\), with \(\mathcal C_I(0)=0\), such that
\begin{equation}\label{eq:aihp-common-domain-seed}
\log\det\bigl(\operatorname{Id}+tT_{I,L}\bigr)
=
L\mathcal L_I(t)
+\mathcal C_I(\log(1+t))
+o(1)
\end{equation}
locally uniformly for \(t\in U:=\mathbb C\setminus(-\infty,-1]\), where \(\log(1+t)\) is the branch vanishing at \(t=0\).
Since \(-1\notin U\), the next proposition extends \eqref{eq:aihp-common-domain-seed} to \(\Omega_I\), thereby reaching the endpoint \(t=-1\).

\begin{proposition}
\label{prop:endpoint-determinant-general}
There exists a unique holomorphic function \(D_I\) on \(\Omega_I\), with \(D_I(0)=0\), such that
\begin{equation}\label{eq:endpoint-determinant-expansion}
\log\det\bigl(\operatorname{Id}+tT_{I,L}\bigr)
=
L\mathcal L_I(t)+D_I(t)+o(1),
\qquad L\to\infty,
\end{equation}
locally uniformly for \(t\in\Omega_I\), where the logarithm is the holomorphic branch vanishing at \(t=0\).
Moreover, \(D_I(t)=\mathcal C_I(\log(1+t))\) for \(t\in U\), and \(D_I(t)\in\mathbb R\) for \(t\in(-\rho_I^{-1},\infty)\).
\end{proposition}

\begin{proof}
Set
\[
E_{I,L}(t)
:=
e^{-L\mathcal L_I(t)}
\det\bigl(\operatorname{Id}+tT_{I,L}\bigr).
\]
By \eqref{eq:aihp-common-domain-seed}, \(E_{I,L}\) already converges locally uniformly on \(U\).
To extend this convergence to \(\Omega_I\) by Vitali's theorem, it remains to establish local boundedness of the tail family \(\{E_{I,L}\}\) on \(\Omega_I\).

Let \(\widetilde W_{I,L}\) be the integral operator on \(L^2(0,\infty)\) with kernel
\[
\widetilde W_{I,L}(\xi,\eta)
=
\frac{\sqrt{q_I(\xi)q_I(\eta)}}{2\pi}
\int_0^L e^{-\mathrm i(\xi-\eta)x}\,\mathrm dx,
\qquad \xi,\eta>0,
\]
and set
\[
\widetilde E_{I,L}(t)
:=
e^{-L\mathcal L_I(t)}
\det\bigl(\operatorname{Id}+t\widetilde W_{I,L}\bigr).
\]
We first bound \(\widetilde E_{I,L}\), and then transfer the bound to \(E_{I,L}\) using the trace-class comparison between \(T_{I,L}\) and \(\widetilde W_{I,L}\).

Extend \(q_I\) by zero to \((-\infty,0]\), and let \(W_L(\sigma)\) denote the truncated Wiener--Hopf operator on \(L^2(0,L)\) with bounded symbol \(\sigma\), in the Fourier-multiplier sense.
In particular,
\[
(W_L(q_I)g)(x')
=
\frac{1}{2\pi}
\int_0^L\int_{\mathbb R}
q_I(\xi)e^{\mathrm i\xi(x'-x)}g(x)\,
\mathrm d\xi\,\mathrm dx,
\qquad 0<x'<L.
\]
For \(g\in L^2(0,L)\), define
\[
(V_{I,L}g)(\xi)
=
\sqrt{\frac{q_I(\xi)}{2\pi}}
\int_0^L e^{-\mathrm i\xi x}g(x)\,\mathrm dx .
\]
Since \(\|V_{I,L}\|_{\mathfrak S_2}^2 =(L/2\pi)\int_0^\infty q_I(\xi)\,\mathrm d\xi<\infty\), we have \(V_{I,L}\in\mathfrak S_2\), and direct computation gives
\[
W_L(q_I)=V_{I,L}^*V_{I,L},
\qquad
\widetilde W_{I,L}=V_{I,L}V_{I,L}^*.
\]
Hence, by Sylvester's identity,
\[
\det\bigl(\operatorname{Id}+t\widetilde W_{I,L}\bigr)
=
\det\bigl(\operatorname{Id}+tW_L(q_I)\bigr).
\]

Fix \(K\Subset\Omega_I\).
By the definition of \(\Omega_I\), the compact set
\[
\{1+tu:t\in K,\ 0\leq u\leq\rho_I\}
\]
is disjoint from \((-\infty,0]\).
In particular,
\[
m_K
:=
\inf_{\substack{t\in K\\0\leq u\leq\rho_I}}
|1+tu|
>0.
\]
Set \(\psi_{I,t}:=\log(1+tq_I)\), where the principal logarithm is used.
Then \(\operatorname{Id}+tW_L(q_I)=W_L(e^{\psi_{I,t}})\), and
\[
\operatorname{tr}W_L(\psi_{I,t})
=
\frac{L}{2\pi}
\int_{\mathbb R}\psi_{I,t}(\xi)\,\mathrm d\xi
=
L\mathcal L_I(t).
\]

These identities put \(\widetilde E_{I,L}\) in the normalized Wiener--Hopf form used in the proof of \cite[Lemma~5.8]{AiFangHou2026}.
We verify the symbol estimates needed there.
Since the zero extension of \(q_I\) belongs to \(L^1(\mathbb R)\cap W^{1,2}(\mathbb R)\), uniformly for \(t\in K\),
\[
|\psi_{I,t}|\leq C_Kq_I,
\qquad
|\psi_{I,t}'|\leq C_K|q_I'|.
\]
Thus \(\{\psi_{I,t}:t\in K\}\) is bounded in \(L^1(\mathbb R)\cap W^{1,2}(\mathbb R)\cap L^\infty(\mathbb R)\).
For \(\lambda\) in any fixed closed disk containing \([0,1]\), the chain and product rules give the same uniform \(L^1\cap W^{1,2}\cap L^\infty\) bounds for \(e^{\lambda\psi_{I,t}}-1\) and \((e^{\lambda\psi_{I,t}}-1)\psi_{I,t}\).
By Plancherel's theorem and the Cauchy--Schwarz inequality, these bounds give the Fourier-side estimates required in \cite[Lemma~5.7]{AiFangHou2026}; the \(L^\infty\)-bound also gives the uniform operator-norm control of the exponential factor used in the interpolation.
Hence the \(\lambda\)-interpolation argument in the proof of \cite[Lemma~5.8]{AiFangHou2026}, with \(\psi_\theta\) replaced by \(\psi_{I,t}\), yields
\[
\limsup_{L\to\infty}
\sup_{t\in K}|\widetilde E_{I,L}(t)|<\infty.
\]

We next transfer this bound to \(E_{I,L}\).
Set \(a=\beta/\alpha\).
Under the unitary dilation on \(L^2(0,\infty)\), the pair \((T_{I,L},\widetilde W_{I,L})\) is unitarily equivalent to \((T_{[1,a],L/\alpha},\widetilde W_{[1,a],L/\alpha})\).
Hence \cite[Proposition~5.5]{AiFangHou2026} gives
\[
\limsup_{L\to\infty}
\|T_{I,L}-\widetilde W_{I,L}\|_{\mathfrak S_1}<\infty.
\]
The factorization above and Plancherel's theorem also give \(0\leq\widetilde W_{I,L}\leq\rho_I\operatorname{Id}\).
Therefore, for \(t\in K\),
\[
\bigl\|
(\operatorname{Id}+t\widetilde W_{I,L})^{-1}
\bigr\|
\leq m_K^{-1}.
\]
By multiplicativity of the Fredholm determinant,
\[
\det\bigl(\operatorname{Id}+tT_{I,L}\bigr)
=
\det\bigl(\operatorname{Id}+t\widetilde W_{I,L}\bigr)
\det\!\left(
\operatorname{Id}
+t(\operatorname{Id}+t\widetilde W_{I,L})^{-1}
(T_{I,L}-\widetilde W_{I,L})
\right).
\]
The trace norm of the perturbation in the second determinant is therefore bounded by \(C_K\), uniformly for \(t\in K\) and all sufficiently large \(L\).
Using \(|\det(\operatorname{Id}+A)| \leq e^{\|A\|_{\mathfrak S_1}}\), we obtain
\[
\limsup_{L\to\infty}\sup_{t\in K}|E_{I,L}(t)|<\infty.
\]
Thus the tail family \(\{E_{I,L}\}\) is locally bounded on \(\Omega_I\).

By \eqref{eq:aihp-common-domain-seed}, \(E_{I,L}(t)\to\exp\{\mathcal C_I(\log(1+t))\}\) locally uniformly on \(U\).
Applying Vitali's theorem to an arbitrary sequence \(L_n\to\infty\), and then the identity theorem, shows that \(E_{I,L}\) converges locally uniformly on \(\Omega_I\) to a holomorphic function \(E_I\), uniquely determined by
\[
E_I(t)=\exp\{\mathcal C_I(\log(1+t))\},
\qquad t\in U.
\]

It remains to pass to logarithms.
By Lemma~\ref{lem:strict-spectral-bound}, each \(E_{I,L}\) is zero-free on \(\Omega_I\); since \(E_I(0)=1\), Hurwitz's theorem shows that \(E_I\) is zero-free as well.
As \(\Omega_I\) is simply connected, let \(D_I\) be the unique holomorphic logarithm of \(E_I\) with \(D_I(0)=0\), and let \(G_{I,L}\) be the normalized holomorphic logarithm of \(E_{I,L}\).
Then
\[
G_{I,L}(t)
=
\log\det\bigl(\operatorname{Id}+tT_{I,L}\bigr)
-L\mathcal L_I(t).
\]
Since \(G_{I,L}'=E_{I,L}'/E_{I,L}\to E_I'/E_I=D_I'\) locally uniformly on \(\Omega_I\), integration over the compact radial hull of any \(K\Subset\Omega_I\) gives \(G_{I,L}\to D_I\) locally uniformly.
This proves \eqref{eq:endpoint-determinant-expansion}.

On \(U\), comparison with \eqref{eq:aihp-common-domain-seed} gives \(D_I(t)=\mathcal C_I(\log(1+t))\).
Finally, for real \(t>-\rho_I^{-1}\), both \(e^{-L\mathcal L_I(t)}\) and \(\det(\operatorname{Id}+tT_{I,L})\) are positive by Lemma~\ref{lem:strict-spectral-bound}.
Hence \(E_{I,L}(t)>0\); passing to the limit and using the zero-freeness of \(E_I\) gives \(E_I(t)>0\).
Thus \(\operatorname{Im}D_I(t)\in2\pi\mathbb Z\).
Since \(D_I(0)=0\) and \((-\rho_I^{-1},\infty)\) is connected, continuity gives \(D_I(t)\in\mathbb R\) for \(t\in(-\rho_I^{-1},\infty)\).
\end{proof}

\begin{proof}[Proof of Theorem~\ref{thm:endpoint-expansion}]
Set \(z=1+t\) in \eqref{eq:general-interval-fredholm}.
The zero-free assertion follows from Lemma~\ref{lem:strict-spectral-bound}, and the remaining assertions follow from Proposition~\ref{prop:endpoint-determinant-general}.
\end{proof}

\begin{proof}[Proof of Theorem~\ref{thm:fixed-counts}: the case \(k=0\)]
Since \(0<\rho_I<1\), we have \(-1\in\Omega_I\).
Evaluating Theorem~\ref{thm:endpoint-expansion} at \(t=-1\) gives
\[
\mathbb P\{N_I(L)=0\}
=
\exp\{L\mathcal L_I(-1)+D_I(-1)+o(1)\}.
\]
By \eqref{eq:endpoint-rate} and \(e^{o(1)}=1+o(1)\), this proves the assertion for \(k=0\).
\end{proof}

\subsection{Endpoint generating functions and fixed-count asymptotics}
\label{subsec:endpoint-neighborhood}

We now normalize the probability generating function by the hole probability and transfer the endpoint expansion in Theorem~\ref{thm:endpoint-expansion} to a neighborhood of \(z=0\).
Coefficient extraction from the normalized generating function, together with the hole asymptotic, then yields the fixed-count probabilities.

Since \(\mathbb P\{N_I(L)=0\}>0\), define
\[
R_{I,L}(z)
:=
\frac{\mathbb E z^{N_I(L)}}
{\mathbb P\{N_I(L)=0\}},\qquad 
B_I(z)
:=
D_I(z-1)-D_I(-1).
\]
For \(|z|<\rho_I^{-1}-1\), we have \(z-1\in\Omega_I\).
Hence \(R_{I,L}\) is zero-free and \(B_I\) is holomorphic on this disk, with \(R_{I,L}(0)=1\) and \(B_I(0)=0\).

For the expansion below, we first record the holomorphic extension of \(A_I\) to the same disk.
Namely, the function \(A_I\) defined in \eqref{eq:endpoint-A} satisfies
\[
A_I(z)
=
\mathcal L_I(z-1)-\mathcal L_I(-1)
=
\frac{1}{2\pi}
\int_0^\infty
\log\left(
1+z\frac{q_I(\xi)}{1-q_I(\xi)}
\right)\,\mathrm d\xi,
\]
where the logarithm is the branch normalized to vanish at \(z=0\).
Since \(q_I/(1-q_I)\in L^1(0,\infty)\cap L^\infty(0,\infty)\), the integral converges locally uniformly on this disk.
Since \(A_I(0)=0\) and \(A_I'(0)=\kappa_I\), holomorphy at the origin gives, as \(z\to0\),
\begin{equation}\label{eq:endpoint-taylor}
A_I(z)=\kappa_I z+O(z^2),
\qquad
zA_I'(z)=\kappa_I z+O(z^2),
\end{equation}
and
\begin{equation}\label{eq:endpoint-variance-taylor}
zA_I'(z)+z^2A_I''(z)
=
\kappa_I z+O(z^2).
\end{equation}

\begin{proposition}
\label{prop:endpoint-neighborhood}
Let \(\log R_{I,L}\) be the holomorphic branch with value \(0\) at \(z=0\).
Then, for \(j=0,1,2\),
\[
\left(z\frac{\mathrm d}{\mathrm dz}\right)^j
\log R_{I,L}(z)
=
L\left(z\frac{\mathrm d}{\mathrm dz}\right)^j A_I(z)
+
\left(z\frac{\mathrm d}{\mathrm dz}\right)^j B_I(z)
+o(1),
\qquad L\to\infty,
\]
locally uniformly for \(|z|<\rho_I^{-1}-1\).
\end{proposition}

\begin{proof}
For \(|z|<\rho_I^{-1}-1\), both \(z-1\) and \(-1\) belong to \(\Omega_I\).
Set
\[
H_{I,L}(t)
:=
\log\mathbb E(1+t)^{N_I(L)}
-
L\mathcal L_I(t)
-
D_I(t).
\]
By Theorem~\ref{thm:endpoint-expansion}, \(H_{I,L}\) is holomorphic on \(\Omega_I\) and converges to zero locally uniformly there.
Evaluating at \(t=z-1\) and \(t=-1\), and subtracting, gives
\[
\log R_{I,L}(z)
=
LA_I(z)+B_I(z)
+
H_{I,L}(z-1)-H_{I,L}(-1).
\]
Here the difference of the logarithms is the holomorphic logarithm of \(R_{I,L}\) that vanishes at \(z=0\), and hence is the normalized branch \(\log R_{I,L}\).
The last two terms involving \(H_{I,L}\) form a holomorphic function that converges to zero locally uniformly, which proves the case \(j=0\).

Cauchy's estimates imply that its first two derivatives also converge to zero locally uniformly.
Applying \(z\,\mathrm d/\mathrm dz\) once and twice therefore proves the cases \(j=1,2\).
\end{proof}

\begin{proof}[Completion of the proof of Theorem~\ref{thm:fixed-counts}]
Let \(k\geq1\) be fixed.
By the definition of \(R_{I,L}\),
\[
\mathbb P\{N_I(L)=k\}
=
\mathbb P\{N_I(L)=0\}[z^k]R_{I,L}(z),
\]
where \([z^k]R_{I,L}(z)\) denotes the coefficient of \(z^k\) in the Taylor expansion of \(R_{I,L}\) at \(z=0\).
Proposition~\ref{prop:endpoint-neighborhood} gives
\[
\log R_{I,L}(w/L)
=
LA_I(w/L)+B_I(w/L)+o(1).
\]
By \eqref{eq:endpoint-taylor}, \(LA_I(w/L)=\kappa_Iw+o(1)\), while \(B_I(w/L)=o(1)\) since \(B_I\) is holomorphic near the origin and \(B_I(0)=0\).
Hence
\[
\log R_{I,L}(w/L)
=
\kappa_Iw+o(1)
\]
locally uniformly for \(w\in\mathbb C\).
Hence \( R_{I,L}(w/L)\longrightarrow e^{\kappa_Iw} \) locally uniformly.
Cauchy's integral formula for coefficients yields
\[
L^{-k}[z^k]R_{I,L}(z)
=
[w^k]R_{I,L}(w/L)
\longrightarrow
\frac{\kappa_I^k}{k!}.
\]
Therefore,
\[
[z^k]R_{I,L}(z)
=
\frac{(L\kappa_I)^k}{k!}(1+o(1)).
\]
Combining this with the already proved case \(k=0\) completes the proof.
\end{proof}

\section{Vanishing-density asymptotics}
\label{sec:vanishing-density}

In this section, we prove Theorem~\ref{thm:vanishing-density}.
We introduce an exponential tilt of \(N_I(L)\), construct the exact finite-\(L\) saddle \(r_{L,n}\), and derive the corresponding point-probability asymptotic using the local central limit theorem for determinantal point processes.
We then compare \(r_{L,n}\) with the limiting saddle \(s_{L,n}\) and identify the exponent through the variational formula for \(J_I\).

\subsection{Exponential tilting and finite-\texorpdfstring{\(L\)}{L}
saddle points}
\label{subsec:tilting-finite-saddles}

We first identify a determinantal representation of the exponentially tilted count and record exact and endpoint asymptotics for its mean and variance.

Let \((\lambda_{j,L})_{j\geq1}\) be the eigenvalues of \(T_{I,L}\), repeated with multiplicity.
By \eqref{eq:general-interval-fredholm} and Lemma~\ref{lem:strict-spectral-bound}, \( 0\leq\lambda_{j,L}\leq\rho_I<1, \) and hence
\[
R_{I,L}(z)
=
\frac{\mathbb E z^{N_I(L)}}
{\mathbb P\{N_I(L)=0\}}
=
\frac{
\det\bigl(\operatorname{Id}+(z-1)T_{I,L}\bigr)
}{
\det\bigl(\operatorname{Id}-T_{I,L}\bigr)
}
=
\prod_{j\geq1}
\left(
1+z\frac{\lambda_{j,L}}{1-\lambda_{j,L}}
\right).
\]
By the Bernoulli representation in \cite[Theorem~7]{HoughEtAl2006} and the identification of the common nonzero spectrum in \cite[Lemmas~2.5--2.6]{AiFangHou2026}, \(N_I(L)\) is distributed as an at most countable sum of independent Bernoulli variables with parameters \((\lambda_{j,L})_{j\geq1}\).

For \(r>0\), define the exponentially tilted measure by
\begin{equation}
\frac{\mathrm d\mathbb Q_{L,r}}{\mathrm d\mathbb P}
:=
\frac{r^{N_I(L)}}{\mathbb E[r^{N_I(L)}]}
=
\frac{r^{N_I(L)}}
{R_{I,L}(r)\mathbb P\{N_I(L)=0\}}.
\label{eq:exponential-tilt}
\end{equation}

\begin{lemma}
\label{lem:tilted-determinantal-law}
For \(r>0\), under \(\mathbb Q_{L,r}\), the random variable \(N_I(L)\) has the same distribution as an at most countable sum of independent Bernoulli variables with parameters
\[
p_{j,L}^{(r)}
=
\frac{r\lambda_{j,L}}
{1-\lambda_{j,L}+r\lambda_{j,L}},
\qquad j\geq1.
\]
Moreover, these parameters are the eigenvalues of the positive trace-class contraction
\[
T_{I,L}^{(r)}
:=
rT_{I,L}
\bigl(\operatorname{Id}+(r-1)T_{I,L}\bigr)^{-1}.
\]
Consequently, under \(\mathbb Q_{L,r}\), \(N_I(L)\) has the same distribution as the total number of points of a determinantal point process whose correlation operator is \(T_{I,L}^{(r)}\).
\end{lemma}

\begin{proof}
By \eqref{eq:exponential-tilt} and the product representation of \(R_{I,L}\),
\[
\mathbb E_{\mathbb Q_{L,r}}z^{N_I(L)}
=
\frac{R_{I,L}(rz)}{R_{I,L}(r)}
=
\prod_{j\geq1}
\frac{
1+rz\lambda_{j,L}/(1-\lambda_{j,L})
}{
1+r\lambda_{j,L}/(1-\lambda_{j,L})
}
=
\prod_{j\geq1}
\left(1+(z-1)p_{j,L}^{(r)}\right).
\]
Since \(p_{j,L}^{(r)}\leq r\lambda_{j,L}/(1-\rho_I)\) and \(\sum_j\lambda_{j,L}<\infty\), we have \(\sum_jp_{j,L}^{(r)}<\infty\).
Thus the right-hand side is the probability generating function of an at most countable sum of independent Bernoulli variables with parameters \((p_{j,L}^{(r)})_{j\geq1}\).

Since \(0\leq\lambda_{j,L}\leq\rho_I<1\), functional calculus shows that \(T_{I,L}^{(r)}\) is positive and has eigenvalues \(p_{j,L}^{(r)}\in[0,1)\).
The preceding bound also shows that \(T_{I,L}^{(r)}\) is trace class, and hence it is a positive trace-class contraction.
By the standard existence criterion for determinantal point processes, \(T_{I,L}^{(r)}\) is the correlation operator of a determinantal point process; see \cite[Theorem~22]{HoughEtAl2006}.
By \cite[Theorem~7]{HoughEtAl2006}, its total number of points has the Bernoulli representation above, proving the final assertion.
\end{proof}

We next record the mean and variance of \(N_I(L)\) under the exponentially tilted measure, together with their endpoint asymptotics, which determine the exact saddle point and the Gaussian prefactor.

\begin{lemma}
\label{lem:tilted-mean-variance}
For \(r>0\), set \(M_{I,L}(r):=\mathbb E_{\mathbb Q_{L,r}}N_I(L)\) and \(V_{I,L}(r):=\operatorname{Var}_{\mathbb Q_{L,r}}N_I(L)\).
Then
\begin{align}
M_{I,L}(r)
&=
r\frac{\mathrm d}{\mathrm dr}\log R_{I,L}(r)
=
\sum_{j\geq1}
\frac{r\lambda_{j,L}}
{1-\lambda_{j,L}+r\lambda_{j,L}},
\label{eq:exact-tilted-mean}\\
V_{I,L}(r)
&=
r\frac{\mathrm d}{\mathrm dr}M_{I,L}(r)
=
\sum_{j\geq1}
\frac{r\lambda_{j,L}(1-\lambda_{j,L})}
{(1-\lambda_{j,L}+r\lambda_{j,L})^2}.
\label{eq:exact-tilted-variance}
\end{align}
Moreover, there exists \(r_0>0\) such that, as \(L\to\infty\),
\begin{align}
M_{I,L}(r)
&=
L\kappa_Ir+O(Lr^2)+O(r)+o(1),
\label{eq:endpoint-tilted-mean}\\
V_{I,L}(r)
&=
L\kappa_Ir+O(Lr^2)+O(r)+o(1),
\label{eq:endpoint-tilted-variance}
\end{align}
uniformly for \(0<r\leq r_0\).
\end{lemma}

\begin{proof}
By Lemma~\ref{lem:tilted-determinantal-law}, under \(\mathbb Q_{L,r}\), the count \(N_I(L)\) has the same distribution as a sum of independent Bernoulli variables with parameters \((p_{j,L}^{(r)})_{j\geq1}\).
Summing their means and variances gives the series in \eqref{eq:exact-tilted-mean} and \eqref{eq:exact-tilted-variance}.
The logarithmic derivative identities follow directly from \eqref{eq:exponential-tilt}.

Choose \(r_0>0\) sufficiently small.
By Proposition~\ref{prop:endpoint-neighborhood} with \(j=1,2\), as \(L\to\infty\),
\begin{align*}
M_{I,L}(r)
&=
LrA_I'(r)+rB_I'(r)+o(1),\\
V_{I,L}(r)
&=
L\bigl(rA_I'(r)+r^2A_I''(r)\bigr)
+rB_I'(r)+r^2B_I''(r)+o(1),
\end{align*}
uniformly for \(0<r\leq r_0\).
Since \(B_I\) is holomorphic near \(0\), the terms involving \(B_I\) are \(O(r)\) uniformly on this interval.
The expansions \eqref{eq:endpoint-taylor} and \eqref{eq:endpoint-variance-taylor} therefore give \eqref{eq:endpoint-tilted-mean} and \eqref{eq:endpoint-tilted-variance}.
\end{proof}

Throughout the remainder of this section, let \(b_L,\varepsilon_L>0\) satisfy
\[
b_L\to\infty,
\qquad
\varepsilon_L\to0,
\qquad
b_L\leq\varepsilon_L L
\]
for all sufficiently large \(L\).

The next proposition determines the asymptotic scales of the exact saddle and the corresponding tilted variance.

\begin{proposition}
\label{prop:endpoint-saddle-variance}
For all sufficiently large \(L\) and every integer \(b_L\leq n\leq\varepsilon_L L\), the equation \(M_{I,L}(r)=n\) has a unique solution \(r_{L,n}>0\).
Moreover, as \(L\to\infty\), uniformly over integers \(b_L\leq n\leq\varepsilon_L L\),
\[
r_{L,n}\sim\frac{n}{L\kappa_I},
\qquad
V_{I,L}(r_{L,n})\sim n.
\]
\end{proposition}

\begin{proof}
For each \(L\), set \(\mathcal N_L:=\mathbb Z\cap[b_L,\varepsilon_LL]\).
If \(\mathcal N_L=\varnothing\), there is nothing to prove, so assume otherwise.
Fix \(\delta\in(0,1)\) and, for \(n\in\mathcal N_L\), let \(r_{L,n}^{\pm}:=(1\pm\delta)n/(L\kappa_I)\).
Since
\[
\sup_{n\in\mathcal N_L}r_{L,n}^{+}
\leq\frac{(1+\delta)\varepsilon_L}{\kappa_I}
\longrightarrow0,
\]
\eqref{eq:endpoint-tilted-mean} gives
\[
M_{I,L}(r_{L,n}^{\pm})
=
(1\pm\delta)n
+
O\left(\frac{n^2}{L}\right)
+
O\left(\frac{n}{L}\right)
+
o(1)
\]
uniformly for \(n\in\mathcal N_L\).
After division by \(n\), the three remainder terms are \(O(\varepsilon_L)\), \(O(L^{-1})\), and \(o(1)/b_L\), respectively, and hence are \(o(1)\) uniformly.
Therefore, for all sufficiently large \(L\),
\[
M_{I,L}(r_{L,n}^{-})<n<M_{I,L}(r_{L,n}^{+}),
\qquad n\in\mathcal N_L.
\]
By \eqref{eq:exact-tilted-variance},
\[
M_{I,L}'(r)=\frac{V_{I,L}(r)}{r}>0,
\qquad r>0,
\]
since \(\mathbb E N_I(L)=\sum_j\lambda_{j,L}>0\) and \(0\leq\lambda_{j,L}<1\).
Thus \(M_{I,L}\) is strictly increasing.
By continuity, for every \(n\in\mathcal N_L\) there is a unique \(r_{L,n}\in(r_{L,n}^{-},r_{L,n}^{+})\) such that \(M_{I,L}(r_{L,n})=n\).
Since \(\delta\) is arbitrary, \(r_{L,n}\sim n/(L\kappa_I)\) uniformly for \(n\in\mathcal N_L\).

In particular, \(r_{L,n}=O(n/L)\) uniformly, so \(\sup_{n\in\mathcal N_L}r_{L,n}\to0\).
Moreover,
\[
\frac{Lr_{L,n}^2}{n}
=
O\left(\frac nL\right)
=
O(\varepsilon_L)=o(1),
\qquad
\frac{r_{L,n}}{n}=O(L^{-1})=o(1)
\]
uniformly for \(n\in\mathcal N_L\), while the uniform \(o(1)\) term in \eqref{eq:endpoint-tilted-variance} is \(o(n)\) because \(n\geq b_L\to\infty\).
Hence
\[
V_{I,L}(r_{L,n})
=
L\kappa_Ir_{L,n}+o(n)
=
n(1+o(1))
\]
uniformly for \(n\in\mathcal N_L\).
\end{proof}

At the exact saddle, the tilted count is centered at the target \(n\).
Combining the local central limit theorem for determinantal point processes with the exact change of measure yields the following point-probability asymptotic.

\begin{proposition}
\label{prop:exact-saddle-point-probability}
As \(L\to\infty\), uniformly over integers \(b_L\leq n\leq\varepsilon_L L\),
\[
\mathbb P\{N_I(L)=n\}
=
\frac{
\exp\!\left\{
-LJ_I(0)+L A_I(r_{L,n})
+D_I(r_{L,n}-1)-n\log r_{L,n}
\right\}
}{
\sqrt{2\pi V_{I,L}(r_{L,n})}
}
(1+o(1)).
\]
\end{proposition}

\begin{proof}
For each \(L\), set \(\mathcal N_L:=\mathbb Z\cap[b_L,\varepsilon_LL]\).
The assertion is vacuous when \(\mathcal N_L=\varnothing\).
Otherwise, by \eqref{eq:exponential-tilt}, with \(r=r_{L,n}\),
\begin{equation}
\label{eq:exact-saddle-change-of-measure}
\mathbb P\{N_I(L)=n\}
=
\mathbb E\!\left[r_{L,n}^{N_I(L)}\right]\,
r_{L,n}^{-n}\,
\mathbb Q_{L,r_{L,n}}\{N_I(L)=n\},
\qquad n\in\mathcal N_L.
\end{equation}

By Lemma~\ref{lem:tilted-determinantal-law}, under \(\mathbb Q_{L,r_{L,n}}\), \(N_I(L)\) has the same distribution as the total number of points of a determinantal point process whose correlation operator is \(T_{I,L}^{(r_{L,n})}\).
By Lemma~\ref{lem:tilted-mean-variance}, its mean and variance are \(M_{I,L}(r_{L,n})=n\) and \(V_{I,L}(r_{L,n})\), respectively.
Proposition~\ref{prop:endpoint-saddle-variance} gives \(V_{I,L}(r_{L,n})\sim n\) uniformly for \(n\in\mathcal N_L\), and hence
\[
\inf_{n\in\mathcal N_L}V_{I,L}(r_{L,n})\longrightarrow\infty.
\]

To obtain the local estimate uniformly, let \(L_k\to\infty\) and choose arbitrarily \(n_k\in\mathcal N_{L_k}\).
The corresponding standard deviations tend to infinity, while the tilted means equal the integers \(n_k\).
Thus \cite[Theorem~1, Eq.~(2.9)]{ForresterLebowitz2014}, applied with \(x=0\) to the sequence of determinantal point processes with correlation operators \(T_{I,L_k}^{(r_{L_k,n_k})}\), gives the local estimate at their respective means \(n_k\).
Since the sequence was arbitrary, the sequential criterion yields
\begin{equation}
\label{eq:tilted-local-limit}
\mathbb Q_{L,r_{L,n}}\{N_I(L)=n\}
=
\frac{1}{\sqrt{2\pi V_{I,L}(r_{L,n})}}
(1+o(1))
\end{equation}
uniformly for \(n\in\mathcal N_L\).

Moreover, Proposition~\ref{prop:endpoint-saddle-variance} gives \(r_{L,n}=O(\varepsilon_L)\to0\) uniformly.
Since \(-1\in\Omega_I\), the points \(r_{L,n}-1\) lie in a fixed compact subset of \(\Omega_I\) for all sufficiently large \(L\).
Hence Theorem~\ref{thm:endpoint-expansion}, together with \eqref{eq:endpoint-rate} and \eqref{eq:endpoint-A}, yields
\[
\log\mathbb E\!\left[r_{L,n}^{N_I(L)}\right]
=
-LJ_I(0)+L A_I(r_{L,n})
+D_I(r_{L,n}-1)+o(1)
\]
uniformly for \(n\in\mathcal N_L\).
Exponentiating this expansion and substituting it together with \eqref{eq:tilted-local-limit} into \eqref{eq:exact-saddle-change-of-measure} completes the proof.
\end{proof}

\subsection{Limiting saddles and local asymptotics}

We now compare the exact finite-\(L\) saddle \(r_{L,n}\) from Proposition~\ref{prop:endpoint-saddle-variance} with the limiting saddle \(s_{L,n}\) characterized by \eqref{eq:limiting-saddle-equation}.
The following lemma justifies replacing \(r_{L,n}\) by \(s_{L,n}\) in the probability asymptotic of Proposition~\ref{prop:exact-saddle-point-probability}.
For \(s\geq0\), set
\[
v_I(s):=sA_I'(s)+s^2A_I''(s).
\]

\begin{lemma}
\label{lem:saddle-point-replacement}
For \(u>0\), set \(\Psi_{L,n}(u):=L A_I(u)-n\log u+D_I(u-1)\).
Then, as \(L\to\infty\), uniformly over integers \(b_L\leq n\leq\varepsilon_L L\),
\begin{equation}
\label{eq:saddle-point-replacements}
\Psi_{L,n}(r_{L,n})
=
\Psi_{L,n}(s_{L,n})+o(1),
\qquad
V_{I,L}(r_{L,n})
=
L v_I(s_{L,n})(1+o(1)).
\end{equation}
\end{lemma}

\begin{proof}
For each \(L\), set \(\mathcal N_L:=\mathbb Z\cap[b_L,\varepsilon_L L]\).
If \(\mathcal N_L=\varnothing\), there is nothing to prove, so assume otherwise.
We first compare the two saddle points and then use this comparison to prove the two replacements in \eqref{eq:saddle-point-replacements}.
Define \(f_I(s):=sA_I'(s)\).
For \(s>0\), differentiating the integral representation of \(A_I\) gives
\[
f_I'(s)
=
\frac{1}{2\pi}
\int_0^\infty
\frac{q_I(\xi)(1-q_I(\xi))}
{(1-q_I(\xi)+s q_I(\xi))^2}\,\mathrm d\xi
>0.
\]
Thus \(f_I\) is strictly increasing and \(v_I(s)=sf_I'(s)\).
Since \(f_I(s_{L,n})=n/L\leq\varepsilon_L\) and \(f_I(0)=0\), \(s_{L,n}\to0\) uniformly for \(n\in\mathcal N_L\).
By \eqref{eq:endpoint-taylor} and \eqref{eq:endpoint-variance-taylor},
\[
f_I(s)=\kappa_I s+O(s^2), \qquad v_I(s)=\kappa_I s+O(s^2),
\]
as \(s\downarrow0\).
Consequently,
\begin{equation}
\label{eq:limiting-saddle-asymptotics}
s_{L,n}
\sim
\frac{n}{L\kappa_I},
\qquad
L v_I(s_{L,n})
\sim n
\end{equation}
uniformly for \(n\in\mathcal N_L\).

Applying Proposition~\ref{prop:endpoint-neighborhood} with \(j=1\), using \(B_I'(u)=D_I'(u-1)\), and comparing \(M_{I,L}(r_{L,n})=n\) with \(L f_I(s_{L,n})=n\), we obtain
\[
L\bigl(f_I(r_{L,n})-f_I(s_{L,n})\bigr)
+r_{L,n}D_I'(r_{L,n}-1)
+o(1)
=
0
\]
uniformly for \(n\in\mathcal N_L\).
By Proposition~\ref{prop:endpoint-saddle-variance} and \eqref{eq:limiting-saddle-asymptotics}, both saddle points converge uniformly to zero.
Since \(f_I'(0)=\kappa_I>0\), there exists \(c>0\) such that \(f_I'(u)\geq c\) for all sufficiently small \(u\).
Moreover, \(r_{L,n}D_I'(r_{L,n}-1)=O(n/L)\) uniformly.
The mean value theorem therefore gives
\begin{equation}
\label{eq:saddle-difference}
|r_{L,n}-s_{L,n}|
=
O\left(\frac{n}{L^2}\right)+o(L^{-1})
\end{equation}
uniformly for \(n\in\mathcal N_L\).
Dividing \eqref{eq:saddle-difference} by the first relation in \eqref{eq:limiting-saddle-asymptotics} and using \(n\geq b_L\to\infty\), we obtain \(r_{L,n}/s_{L,n}\to1\) uniformly.

Set \(\Phi_{L,n}(u):=L A_I(u)-n\log u\), so that \(\Psi_{L,n}(u)=\Phi_{L,n}(u)+D_I(u-1)\).
By \eqref{eq:limiting-saddle-equation}, \(\Phi_{L,n}'(s_{L,n})=0\).
If \(u\) lies between \(r_{L,n}\) and \(s_{L,n}\), then \(u\asymp s_{L,n}\asymp n/L\) uniformly.
Since \(A_I''\) is bounded near the origin,
\[
\Phi_{L,n}''(u)
=
L A_I''(u)+\frac{n}{u^2}
=
O\left(\frac{L^2}{n}\right).
\]
Taylor's theorem and \eqref{eq:saddle-difference} therefore yield
\[
\Phi_{L,n}(r_{L,n})-\Phi_{L,n}(s_{L,n})
=
O\left(
\frac{L^2}{n}|r_{L,n}-s_{L,n}|^2
\right)
=
o(1)
\]
uniformly for \(n\in\mathcal N_L\), since \(n/L\leq\varepsilon_L\to0\) and \(n\geq b_L\to\infty\).
Moreover, the boundedness of \(D_I'\) near \(-1\) and \eqref{eq:saddle-difference} imply \(D_I(r_{L,n}-1)-D_I(s_{L,n}-1)=o(1)\) uniformly.
Adding these two estimates proves the first relation in \eqref{eq:saddle-point-replacements}.

Applying Proposition~\ref{prop:endpoint-neighborhood} with \(j=2\) and using \(B_I^{(j)}(u)=D_I^{(j)}(u-1)\), \(j=1,2\), gives
\[
V_{I,L}(r_{L,n})
=
L v_I(r_{L,n})
+r_{L,n}D_I'(r_{L,n}-1)
+r_{L,n}^2D_I''(r_{L,n}-1)
+o(1).
\]
By Proposition~\ref{prop:endpoint-saddle-variance}, \(r_{L,n}=O(n/L)\) uniformly.
Since \(D_I'\) and \(D_I''\) are bounded near \(-1\),
\[
r_{L,n}D_I'(r_{L,n}-1)
+
r_{L,n}^2D_I''(r_{L,n}-1)
+
o(1)
=
o(n)
\]
uniformly for \(n\in\mathcal N_L\).
Hence
\[
V_{I,L}(r_{L,n})
=
L v_I(r_{L,n})+o(n)
\]
uniformly for \(n\in\mathcal N_L\).
Since \(v_I'\) is bounded near the origin, \eqref{eq:saddle-difference} also gives
\[
L|v_I(r_{L,n})-v_I(s_{L,n})|
=
O\left(\frac{n}{L}\right)+o(1)
=
o(n).
\]
It follows that \(V_{I,L}(r_{L,n})=L v_I(s_{L,n})+o(n)\).
The second relation in \eqref{eq:limiting-saddle-asymptotics} proves the second relation in \eqref{eq:saddle-point-replacements}.
\end{proof}

\begin{proof}[Proof of Theorem~\ref{thm:vanishing-density}]
For each \(L\), set \(\mathcal N_L:=\mathbb Z\cap[b_L,\varepsilon_L L]\).
If \(\mathcal N_L=\varnothing\), there is nothing to prove, so assume otherwise.
With \(\Psi_{L,n}\) as in Lemma~\ref{lem:saddle-point-replacement}, Proposition~\ref{prop:exact-saddle-point-probability} and \eqref{eq:saddle-point-replacements} give
\[
\mathbb P\{N_I(L)=n\}
=
\frac{
\exp\{-LJ_I(0)+\Psi_{L,n}(s_{L,n})\}
}{
\sqrt{2\pi L v_I(s_{L,n})}
}
(1+o(1))
\]
uniformly for \(n\in\mathcal N_L\).
By \eqref{eq:rate-variational} and \eqref{eq:limiting-saddle-equation},
\[
J_I(n/L)
=
J_I(0)
+\frac{n}{L}\log s_{L,n}
-A_I(s_{L,n}).
\]
By the definition of \(\Psi_{L,n}\), this identity gives
\[
-LJ_I(0)+\Psi_{L,n}(s_{L,n})
=
-LJ_I(n/L)+D_I(s_{L,n}-1).
\]
Substitution proves the probability asymptotic in Theorem~\ref{thm:vanishing-density}.
The remaining two assertions follow from \eqref{eq:limiting-saddle-asymptotics} and the definition of \(v_I\).
\end{proof}


\begin{thebibliography}{99}

\bibitem{AiFangHou2026}
Q. Ai, X. Fang, and S. Hou.
\newblock Hardy--Szeg\H{o} point processes: large deviations and strong
Szeg\H{o} asymptotics.
\newblock \href{https://arxiv.org/abs/2608.17509}{arXiv:2608.17509}, 2026.


\bibitem{BuckleyNishryPeledSodin2018}
J. Buckley, A. Nishry, R. Peled, and M. Sodin.
\newblock Hole probability for zeroes of Gaussian Taylor series with finite
radii of convergence.
\newblock \emph{Probab. Theory Related Fields} \textbf{171} (2018),
no.~1--2, 377--430.

\bibitem{CostinLebowitz1995}
O. Costin and J.~L. Lebowitz.
\newblock Gaussian fluctuation in random matrices.
\newblock \emph{Phys. Rev. Lett.} \textbf{75} (1995), no.~1, 69--72.

\bibitem{FenzlLambert2022}
M. Fenzl and G. Lambert.
\newblock Precise deviations for disk counting statistics of invariant
determinantal processes.
\newblock \emph{Int. Math. Res. Not. IMRN} \textbf{2022} (2022),
no.~10, 7420--7494.

\bibitem{ForresterLebowitz2014}
P.~J. Forrester and J.~L. Lebowitz.
\newblock Local central limit theorem for determinantal point processes.
\newblock \emph{J. Stat. Phys.} \textbf{157} (2014), no.~1, 60--69.

\bibitem{HoughEtAl2006}
J.~B. Hough, M. Krishnapur, Y. Peres, and B. Vir\'ag.
\newblock Determinantal processes and independence.
\newblock \emph{Probab. Surv.} \textbf{3} (2006), 206--229.

\bibitem{HoughKrishnapurPeresVirag2009}
J.~B. Hough, M. Krishnapur, Y. Peres, and B. Vir\'ag.
\newblock \emph{Zeros of Gaussian Analytic Functions and Determinantal
Point Processes}.
\newblock University Lecture Series, vol.~51, American Mathematical
Society, Providence, RI, 2009.

\bibitem{Nishry2010}
A. Nishry.
\newblock Asymptotics of the hole probability for zeros of random entire
functions.
\newblock \emph{Int. Math. Res. Not. IMRN} \textbf{2010} (2010),
no.~15, 2925--2946.

\bibitem{PeresVirag2005}
Y. Peres and B. Vir\'ag.
\newblock Zeros of the i.i.d. Gaussian power series: a conformally invariant
determinantal process.
\newblock \emph{Acta Math.} \textbf{194} (2005), no.~1, 1--35.

\bibitem{Simon2005}
B. Simon.
\newblock \emph{Trace Ideals and Their Applications}, 2nd ed.
\newblock Mathematical Surveys and Monographs, vol.~120,
American Mathematical Society, Providence, RI, 2005.

\bibitem{SodinTsirelson2005}
M. Sodin and B. Tsirelson.
\newblock Random complex zeroes, {III}. Decay of the hole probability.
\newblock \emph{Israel J. Math.} \textbf{147} (2005), 371--379.

\bibitem{Soshnikov2002}
A. Soshnikov.
\newblock Gaussian limit for determinantal random point fields.
\newblock \emph{Ann. Probab.} \textbf{30} (2002), no.~1, 171--187.

\end{thebibliography}
\end{document}